\documentclass{article}
\usepackage{graphicx} 
\usepackage{amsmath} 
\usepackage{amsfonts}
\usepackage{amsthm}
\theoremstyle{definition}

\usepackage{amssymb}
\usepackage{wrapfig}
\usepackage{lmodern}  
\usepackage[none]{hyphenat} 
\usepackage{graphicx} 
\graphicspath{ {./images/} } 
\usepackage{caption}
\usepackage{subcaption}
\usepackage{booktabs}
\usepackage{parskip} 
\usepackage[export]{adjustbox} 
\usepackage{float} 
\usepackage{comment}
\usepackage{bbm}
\usepackage[pdfencoding=auto, psdextra]{hyperref}
\usepackage{lineno}
\usepackage{fullpage}
\usepackage{authblk}

\usepackage{titlesec}

\usepackage{appendix}
\usepackage[svgnames]{xcolor}
\colorlet{LightGray}{gray!20}
\usepackage{minted}
\numberwithin{equation}{section}
\usepackage{graphicx} 
\usepackage{lmodern}

\newtheorem{prop}{Proposition} 
\newtheorem{lem}[prop]{Lemma}
\newtheorem{thm}[prop]{Theorem}

\newtheorem{remark}[prop]{Remark}
\usepackage{newtxtext, newtxmath}
\usepackage[
    natbib=true,
    style=numeric,
    sorting=none
]{biblatex}
\begin{document}
\title{Geometric desingularisation of the sharp-to-smooth travelling wave transition}

\author{Zak Sattar\thanks{z.sattar@sms.ed.ac.uk}}
\affil{School of Mathematics and Maxwell Institute for Mathematical Sciences, University of Edinburgh, James Clerk Maxwell Building, King’s Buildings, Peter Guthrie Tait Road, Edinburgh EH9 3FD, United Kingdom}

\date{\today}

\maketitle


\begin{abstract}
\noindent
We study travelling front solutions of a family of degenerate Fisher-KPP equations $u_t=(u^n u_x)_x+u(1-u^n)$, where $n$ is a positive integer. At the minimal wave speed $c_{\rm min}=\tfrac{1}{\sqrt{1+n}}$, these equations admit  sharp front solutions, corresponding to an explicit heteroclinic orbit in the travelling wave phase-space. We show how this sharp front is perturbed when the wave speed is increased to $c=\tfrac{1}{\sqrt{1+n}}+\varepsilon$, with $\varepsilon$ sufficiently small.  We use a well-motivated geometric desingularisation (also known as blow-up) near the degenerate equilibrium at the leading edge. By analysing the resulting directional and rescaling charts, we construct a singular heteroclinic orbit connecting the relevant asymptotic states, providing a simple geometric proof of the transition from sharp to smooth travelling fronts.

\end{abstract}


\section{Introduction} \label{intro}
In this paper we consider the family of degenerate Fisher-KPP equations
\begin{equation}
    u_t=(u^n u_x)_x+u(1-u^n),
    \label{pde1}
\end{equation}
where $n$ is a positive integer.
Travelling front solutions for Equation~\eqref{pde1} with $n=1$ have been studied extensively \cite{NEWMAN1980325, ARONSON1980161,  sanchez1994existence, SHERRATT199633, Sherratt2010}.
In $1980$, Newman \cite{NEWMAN1980325} and Aronson \cite{ARONSON1980161} independently wrote down an exact travelling wave solution 
\begin{equation}
u(x,t)=
\begin{cases}
1-e^{1/\sqrt{2}\left(x-t/\sqrt{2}\right)},
& x-t/\sqrt{2}<0,\\
0,
& x-t/\sqrt{2}\geq 0,
\end{cases}
\label{exact solution}
\end{equation}
at the minimal wave speed $c_{\rm min}=\tfrac{1}{\sqrt{2}}$ which satisfies the boundary conditions $U(-\infty)=1$ and $U(\infty)=0$. Equation \eqref{exact solution} is non-smooth at $x=\tfrac{t}{\sqrt{2}}$ and is known as a sharp front. 
Further work concerning front solutions to Equation~\eqref{pde1} includes the paper by S{\'a}nchez-Gardu{\~n}o and Maini \cite{sanchez1994existence}, who used a phase-plane argument to show the existence of a minimal wave speed $c_{\rm min}$, for which there are no front solutions for $c<c_{\rm min}$ and smooth front solutions for $c>c_{\rm min}$, and the two papers by Sherratt \cite{SHERRATT199633,Sherratt2010}, who used singular perturbation theory and matched asymptotics to compute an asymptotic approximation to the smooth front for wave speeds $c=c_{\rm min}+\varepsilon$. 
Thus, Sherratt's work identifies the sharp-to-smooth transition and gives a detailed description of the travelling-wave solutions \cite[Equation $(3.15)$]{Sherratt2010}.

The purpose of this paper is different. We will consider \eqref{pde1} for all positive integers $n$.
We prove the existence of smooth fronts, and give a geometric explanation of how the sharp front unfolds into a smooth front when $c=c_{\rm min}+\varepsilon$, where $c_{\rm min}=\tfrac{1}{\sqrt{1+n}}$ and $\varepsilon\ll 1$. 
Examples of sharp and smooth front solutions to \eqref{pde1} are visualised in Figure~\ref{fig: pde sol}.
\begin{figure}
    \centering
    \includegraphics[scale = 0.51]{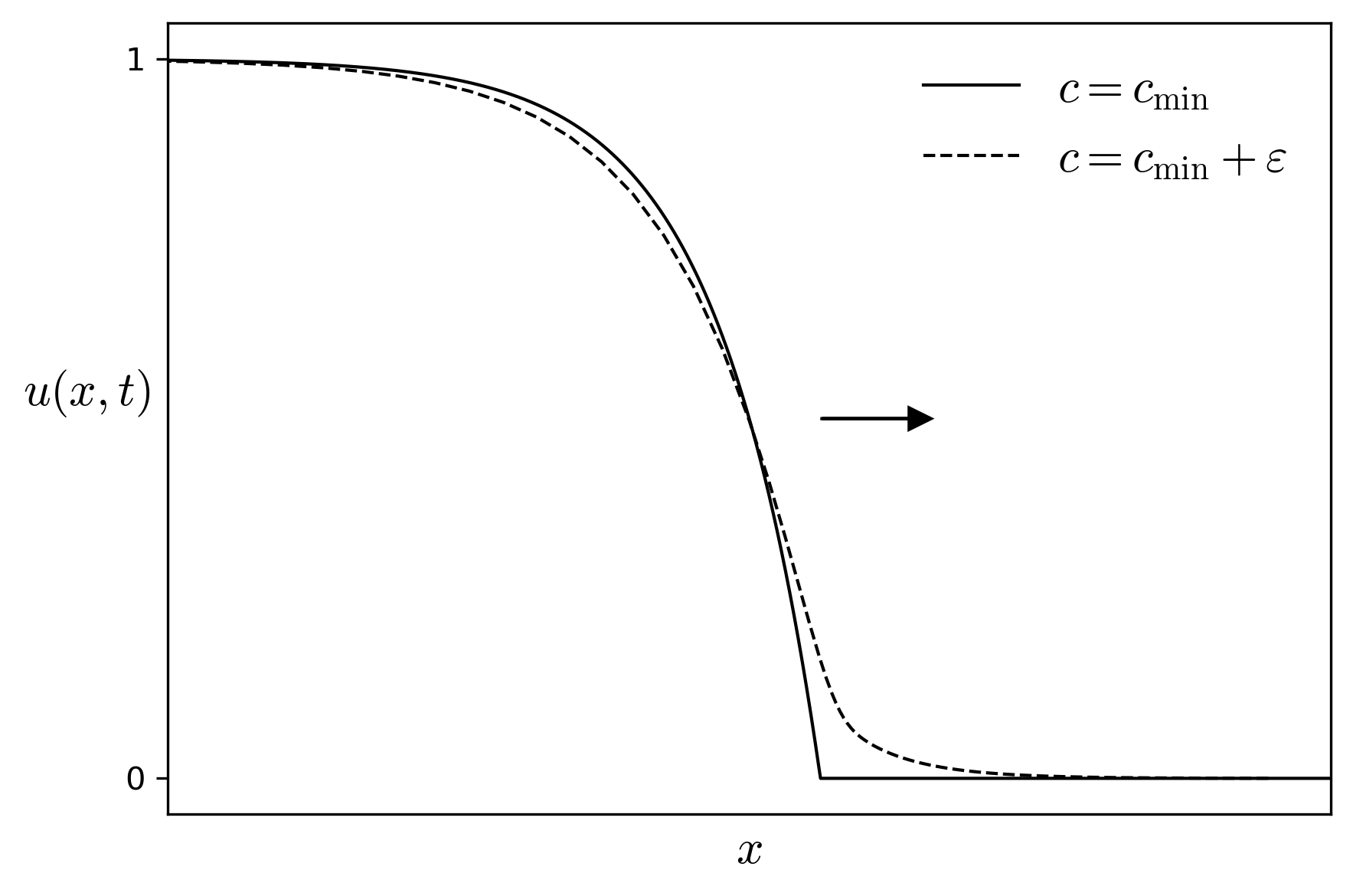}
    \includegraphics[scale = 0.51]{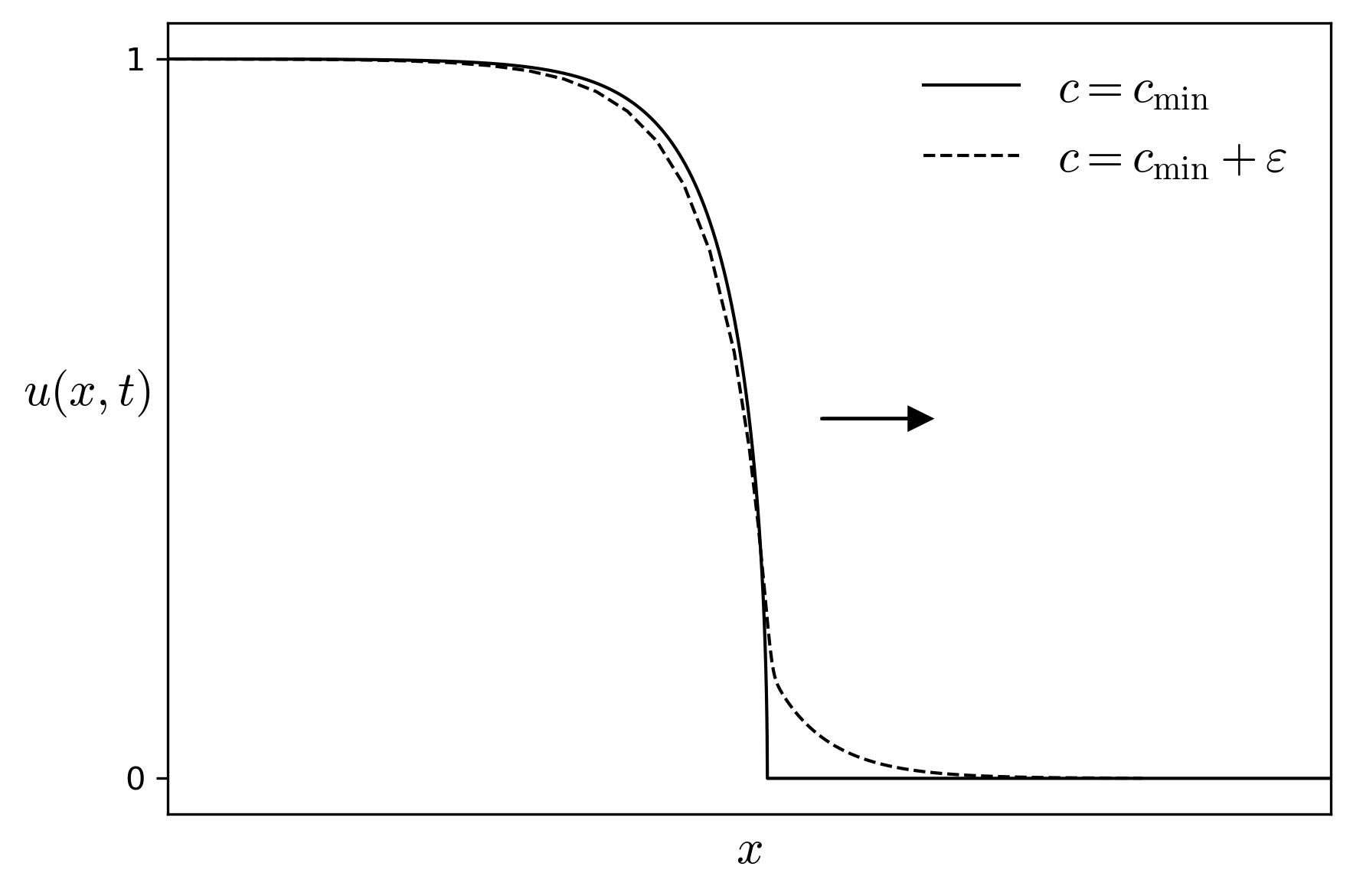}
    \caption{Travelling wave solutions to \eqref{pde1} with $\varepsilon=0.2$. $n=1$ (left), $n=2$ (right).}
    \label{fig: pde sol}
\end{figure}

The perturbation from the minimal speed is singular near the leading edge. In particular, we will show that a regular expansion about the sharp front breaks down when $u(x-ct)=\mathcal{O}(\sqrt[1+n]{\varepsilon})$.
Our proof differs from Sherratt's \cite{Sherratt2010} as we resolve this transition using geometric desingularisation, also known as blow-up. The blow-up method was introduced by Dumortier
and Roussarie \cite{DumortierRoussarie1996} and further developed by Krupa and Szmolyan \cite{KrupaSzmolyan2001}. The method was initially applied to analyse non-hyperbolic points in singular perturbation problems; however, it has since been  applied to other ``singularities", such as non-smooth transitions in vector fields \cite{Dumortier_2007}.

In the blown-up phase-space, the sharp front is replaced by a singular heteroclinic orbit constructed in three coordinate charts. The aim of this paper is to provide a simple geometric proof of the transition from sharp to smooth front solutions for the family of reaction-diffusion equations \eqref{pde1}. The main result of this article is stated below.

\begin{thm}
    For $\varepsilon\in[0, \varepsilon_0)$ with sufficiently small $\varepsilon_0$ and positive integer $n$, the equation \eqref{pde1} admits a smooth monotone front solution connecting $u=1$ and $u=0$, with speed $c=\tfrac{1}{\sqrt{1+n}}+\varepsilon$.
    \label{thm:1}
\end{thm}
\begin{remark}
    The choice $c=\tfrac{1}{\sqrt{1+n}}+\varepsilon$ is made only for convenience. Any perturbation $c=\tfrac{1}{\sqrt{1+n}}+\mu(\varepsilon)$ with $\mu(\varepsilon)\to 0^+$ as $\varepsilon\to0$ yields a smooth front; such a choice merely rescales $\varepsilon$ and leaves the proof of this result unchanged.
\end{remark}

\section{Geometric analysis}
We now perform some preliminary analysis of \eqref{pde1}, before motivating and introducing the geometric desingularisation required to prove Theorem~\ref{thm:1}.

Introducing the travelling wave variable $\eta=x-ct$ and writing $u(x,t)=U(\eta)$, equation \eqref{pde1} becomes
\begin{equation}
    -cU'=(U^nU')'+U(1-U^n).
    \label{tw ode}
\end{equation}
We define $V=U^{n-1}U'$ so that we obtain the system of equations
\begin{equation}
    \begin{aligned}
        U'&=U^{1-n}V,\\
        U^n V'&=-cV-V^2-U^n(1-U^n).
        \label{singular ode}
    \end{aligned}
\end{equation}
We seek solutions which satisfy the boundary conditions $U(-\infty)=1$ and $U(\infty)=0.$ To remove the singularity at $U=0$, we introduce the desingularised variable $\xi$, which is defined through $\tfrac{d}{d\xi}=U^n\tfrac{d}{d\eta}$. Then, system \eqref{singular ode} becomes
\begin{equation}
    \begin{aligned}
        \dot{U} &=UV,\\
        \dot{V} &=-cV-V^2-U^n(1-U^n),
        \label{non-singular ode}
    \end{aligned}
\end{equation}
where the overdot denotes differentiation with respect to the new independent variable $\xi$.

Equation~\eqref{non-singular ode} has the relevant equilibrium points at $(0,0)$, $(0, -c)$, and $(1,0)$. 
For $c=\tfrac{1}{\sqrt{1+n}},$ the sharp front \eqref{exact solution} corresponds to a heteroclinic orbit connecting $(0, -c)$ and $(1,0)$ given by $V(U)=-\tfrac{1}{\sqrt{1+n}}(1-U^n)$.
The sharp front \eqref{exact solution} can then be obtained by integrating $U'=-\tfrac{1}{\sqrt{1+n}}U^{1-n}(1-U^n)$, giving $U(\eta)=\left(1-e^{ \tfrac{n}{\sqrt{1+n}}\eta}\right)^{1/n} $ for $\eta< 0$, and $U(\eta)=0$ for $\eta\geq 0$.

\begin{remark}[Regularity of the fronts]
The sharp  ($\varepsilon=0$) and perturbed ($\varepsilon>0$) fronts differ in regularity. For $\varepsilon=0$ the profile
$U(\eta)$ has compact support and meets $U=0$ at the finite point $\eta=0$. Near $\eta=0^-$ one has $U\sim\left(-\tfrac{n}{\sqrt{1+n}}\eta\right)^{1/n}$,
giving $U'\sim -\tfrac{1}{\sqrt{1+n}}U^{1-n}$, so that for $n=1$ the front is $C^{0, 1}$
with a corner ($U'$ jumps from $-\tfrac{1}{\sqrt{2}}$ to $0$), while for $n\ge2$
the gradient becomes unbounded, $U'(\eta)\to-\infty$ as $\eta\to0^-$, and $U$ is only
$C^{0, 1/n}$. By contrast, for $\varepsilon>0$ the
front is strictly positive for all $\eta$, and is $C^\infty$. The
increase in wave speed regularises the leading edge and removes the sharp
interface for every $n\ge1$.
\end{remark}

We will now study the smooth waves obtained by solving \eqref{singular ode} with $c=\tfrac{1}{\sqrt{1+n}}+\varepsilon$.

To study smooth front solutions to \eqref{pde1} we perform the preliminary scaling $W=V+\tfrac{1}{\sqrt{1+n}}$, which shifts $\left(0, -c\right)$ to $Q^0\equiv(0, -\varepsilon)$ (or $(0,0)$ in the singular limit), then we set $c=\tfrac{1}{\sqrt{1+n}}+\varepsilon$ and append the trivial $\dot{\varepsilon}=0$, so that \eqref{non-singular ode} becomes
\begin{equation}
    \begin{aligned}
        \dot{U} &=U\left(W-\tfrac{1}{\sqrt{1+n}}\right),\\
        \dot{W} &=\tfrac{1}{\sqrt{1+n}}\varepsilon+\left(\tfrac{1}{\sqrt{1+n}}-\varepsilon\right)W-W^2-U^n(1-U^n),\\
        \dot{\varepsilon}&=0.
        \label{ode epsilon}
    \end{aligned}
\end{equation}
For smooth wavefronts, we require
$$U(-\infty)=1,\;W(-\infty)=\tfrac{1}{\sqrt{1+n}} \;\;\text{and}\;\; U(\infty)=0,\;W(\infty)=\tfrac{1}{\sqrt{1+n}}.$$
\begin{remark}
As $W=V+\tfrac{1}{\sqrt{1+n}}$ and $U^{n-1} U'=V,$
we claim that a smooth front satisfies
$V(\infty)=U^{n-1}(\infty)U'(\infty)=0.$ However, for $n>1$, the condition $V(\infty)=0$ alone does not imply $U'(\infty)=0$, since we could have $U^{n-1}(\infty)=0.$ The required decay rate is given by the
centre-manifold expansion near $(U,W)=\left(0, \tfrac{1}{\sqrt{1+n}}\right)$: $V=W-\tfrac{1}{\sqrt{1+n}}=-\tfrac{1}{\tfrac{1}{\sqrt{1+n}}+\varepsilon}U^n+\mathcal{O}(U^{2n}).$
Consequently,
$U'=-\tfrac{1}{\tfrac{1}{\sqrt{1+n}}+\varepsilon}U+\mathcal{O}(U^{n+1}),$
and therefore,  $U'(\infty)=0$ when $U(\infty)=0$.
\end{remark}

Our goal is to construct a singular heteroclinic orbit $\Gamma$, connecting $Q^-\equiv\left(1,\tfrac{1}{\sqrt{1+n}}\right)$ and $Q^+\equiv\left(0,\tfrac{1}{\sqrt{1+n}}\right)$ in blown-up space.
We will now motivate why this cannot be done in ``regular" phase-space.

For $\varepsilon=0$, \eqref{ode epsilon} has an exact heteroclinic orbit which is the union of $W^{\rm u}(Q^-)$ and $W^{\rm s}(Q^0)$. This orbit is simply given by $W(U)=\tfrac{1}{\sqrt{1+n}}U^n$, and corresponds to the sharp front solution of \eqref{pde1}. 
We now consider \eqref{ode epsilon} with sufficiently small $\varepsilon>0$ and introduce the naive ansatz $W(U)=\tfrac{1}{\sqrt{1+n}}U^n+\varepsilon h_1(U)+\mathcal{O}(\varepsilon^{2})$.
We substitute the ansatz into the invariance equation $\dot{W}=\tfrac{dW}{dU}\dot{U}$, and collect terms of order $\varepsilon$ to obtain 
\begin{equation}
    \frac{d}{dU}h_1(U)=\frac{1-U^n+(1-(2+n)U^n)h_1(U)}{U(U^n-1)},
\end{equation}
which we solve with the condition $h_1(1)=0$, giving $h_1(U)=\tfrac{\int_U^1(1-s^n)^{\tfrac{1+n}{n}}ds}{U(1-U^n)^{\tfrac{1+n}{n}}}$. As $U\to 0$, $h_1(U)\to \tfrac{\alpha(n)}{U}$, where
$\alpha(n)\equiv\int_0^1(1-s^n)^{\tfrac{1+n}{n}}ds$. Specifically,  $\alpha(1)=\tfrac{1}{3},\alpha(2)=\tfrac{3\pi}{16},\alpha(3)= \tfrac{2}{15}\tfrac{\Gamma\left(\tfrac{1}{3}\right)^2}{\Gamma\left(\tfrac{2}{3}\right)}$, furthermore, $\alpha(n+1)>\alpha(n)$ and $\alpha(n)\to1$ as $n\to \infty$. 
The perturbed orbit is then given by 
\begin{equation}
    W(U)=\frac{1}{\sqrt{1+n}}U^n+\varepsilon \tfrac{\int_U^1(1-s^n)^{\frac{1+n}{n}}ds}{U(1-U^n)^{\tfrac{1+n}{n}}}+\mathcal{O}(\varepsilon^2).
    \label{perturbed orbit}
\end{equation}
Equation \eqref{perturbed orbit} becomes singular as $U\to0$, and the leading-order terms balance, $\tfrac{1}{\sqrt{1+n}}U^n\sim \varepsilon \tfrac{\alpha(n)}{U}$ when $U=\mathcal{O}(\sqrt[1+n]{\varepsilon})$. As $W\sim U^n$ along the orbit for $\varepsilon=0$, we similarly find $W=\mathcal{O}(\sqrt[1+n]{\varepsilon^n})$.
In Figure~\ref{fig: vectorfields}, we display the phase-plane dynamics of \eqref{ode epsilon} for $n=1$.
In particular, we plot the heteroclinic orbits corresponding to the smooth and sharp fronts in Figure~\ref{fig: pde sol}.

\begin{figure}
    \centering
    \includegraphics[scale = 0.6]{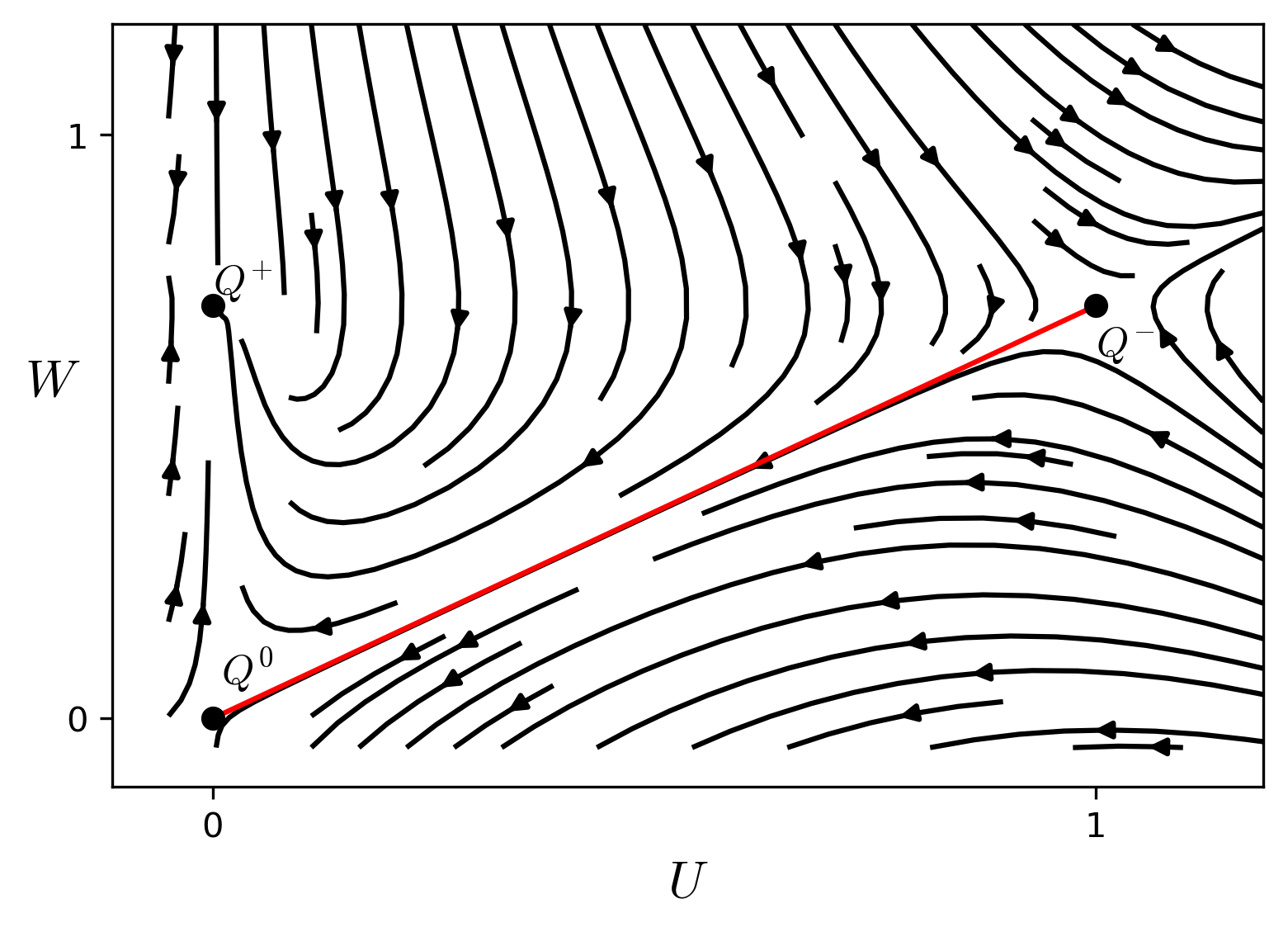}
     \includegraphics[scale = 0.6]{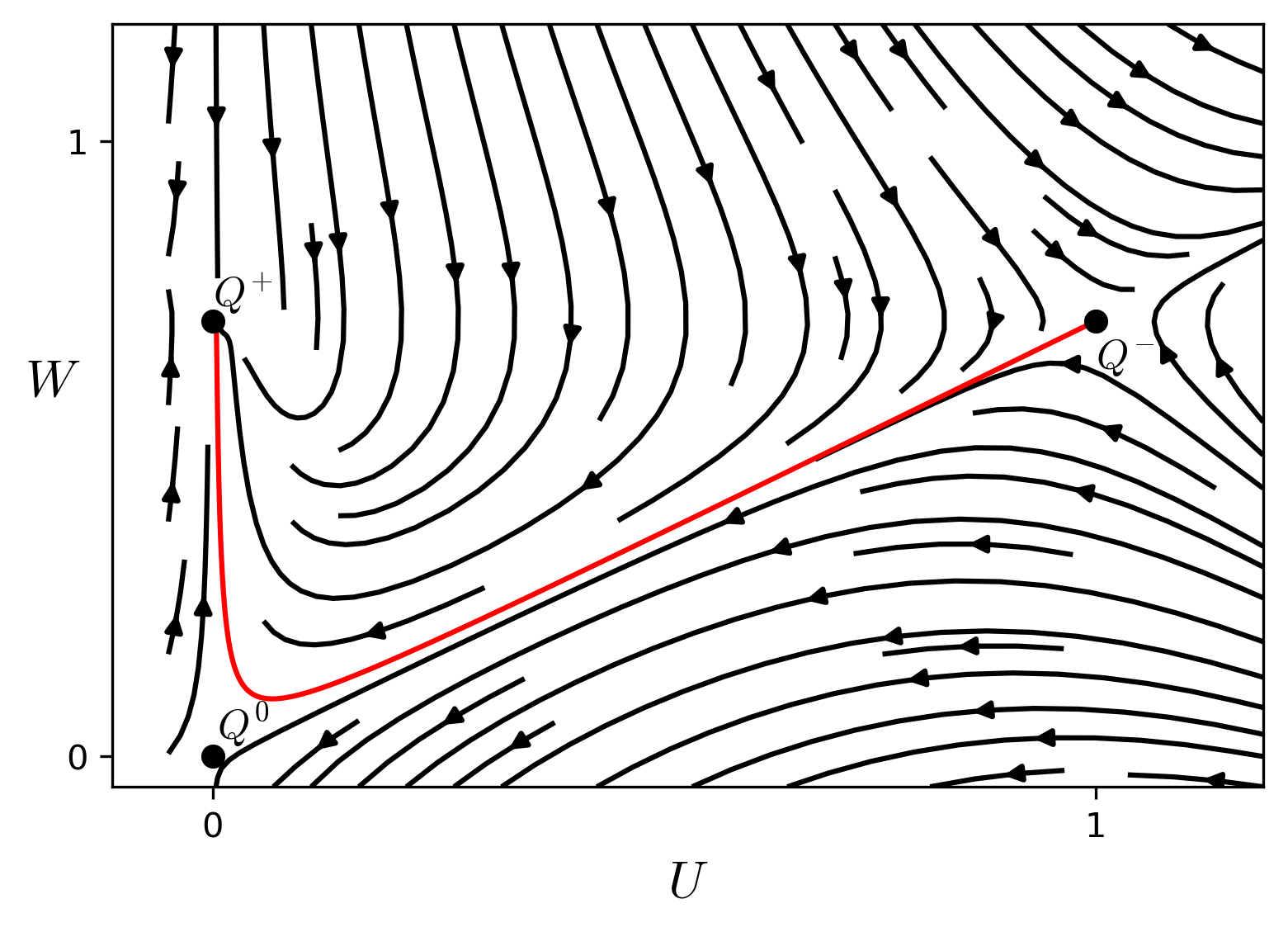}
    \caption{The phase-plane dynamics of \eqref{ode epsilon} for $n=1$, with $\varepsilon=0$ (left) and $\varepsilon=0.01$ (right).}
    \label{fig: vectorfields}
\end{figure}

\subsection{Geometric desingularisation}
As the approximation of $W^{\rm u}(Q^-)$ \eqref{perturbed orbit} breaks down when $U=\mathcal{O}(\sqrt[1+n]{\varepsilon})$ and $W=\mathcal{O}(\sqrt[1+n]{\varepsilon^n})$, we introduce the following blow-up transformation (geometric desingularisation) to resolve the non-smooth transition in a neighbourhood of $(U,W,\varepsilon)=(0,0,0)$
\begin{equation}
    U=\bar{r}\bar{u}, \;\; W=\bar{r}^n\bar{w},\;\; \varepsilon=\bar{r}^{1+n}\bar{\varepsilon}.
    \label{blow-up}
\end{equation}
Here, $(\bar{u}, \bar{w},  \bar{\varepsilon})\in \mathbb{S}_+^2:=\{(\bar{u},\bar{w}, \bar{\varepsilon})\:|\; \bar{u}^2+\bar{w}^2+\bar{\varepsilon}^2=1\}\cap\{\bar{\varepsilon}\geq 0\}$, with $\bar{r}\in[0, r_0]$ for $r_0>0$ sufficiently small.
The blow-up identifies a weighted phase-space scaling which resolves the singular transition and reduces the matching to simple explicit singular orbits.

We require three charts to describe all of the dynamics. Specifically, we need a chart to cover the regimes where $U=W=\mathcal{O}(1)$, then $U=\mathcal{O}(\sqrt[1+n]{\varepsilon}), W=\mathcal{O}(\sqrt[1+n]{\varepsilon^n})$, and finally  $U=\mathcal{O}(\varepsilon)$ and $W=\mathcal{O}(1)$. 
The first chart $K_1$ is defined by setting $\bar{u}=1$, so that \eqref{blow-up} becomes
\begin{equation}
    U=r_1, \;\; W=r_1^nw_1,\;\; \varepsilon=r_1^{1+n}\varepsilon_1,
    \label{blow-up in K1}
\end{equation}
The second chart $K_2$ is defined by setting 
$\bar{\varepsilon}=1$, then \eqref{blow-up} becomes 
\begin{equation}
    U=r_2u_2, \;\; W=r_2^nw_2,\;\; \varepsilon=r_2^{1+n},
    \label{blow-up in K2}
\end{equation}
and the final chart $K_3$ is defined by setting 
$\bar{w}=1$, in which case \eqref{blow-up} becomes 
\begin{equation}
    U=r_3u_3, \;\; W=r_3^n,\;\; \varepsilon=r_3^{1+n}\varepsilon_3.
    \label{blow-up in K3}
\end{equation}
The transition maps between the charts are given by 
\begin{equation}
    \kappa_{12}:K_1\to K_2,\;\;(u_2, w_2, r_2)=\left(\tfrac{1}{\sqrt[1+n]{\varepsilon_1}}, \tfrac{w_1}{\sqrt[1+n]{\varepsilon_1^n}}, r_1\sqrt[1+n]{\varepsilon_1}\right)
    \label{transion map kappa_{12}}
\end{equation}
and \begin{equation}
    \kappa_{23}:K_2\to K_3,\;\;(u_3, r_3, \varepsilon_3)=\left(\tfrac{u_2}{\sqrt[n]{w_2}}, r_2\sqrt[n]{w_2}, \tfrac{1}{\left(\sqrt[n]{w_2}\right)^{1+n}}\right).
    \label{transion map kappa_{23}}
\end{equation}
\begin{remark}
    Sherratt \cite{Sherratt2010} also required three ``rescalings" to resolve the singularity at the leading edge for \eqref{pde1} with $n=1$. In particular, he constructed two outer solutions away from the singularity and a corner solution to approximate the smooth front. However, Sherratt takes $u=\mathcal{O}(\varepsilon)$, whereas for $n=1$, we take $u=\mathcal{O}(\sqrt{\varepsilon})$.
\end{remark}
In each of the following subsections, we will analyse the dynamics in the charts $K_i$, with $i=1,2,3$, and construct the pieces of the singular heteroclinic orbit $\Gamma$. Following this, we will give a simple perturbation argument to show how $\Gamma$ persists for sufficiently small $\varepsilon$.

\begin{remark}
For any object $\Box$ in $(U,W,\varepsilon)$-space, we denote the corresponding blown-up object by $\overline{\Box}$. Moreover, in chart $K_i$, with $i=1,2,3$, that object will be denoted by $\Box_i$.
\end{remark}

\subsection{$U-$directional chart $K_1$}
In this subsection, we study the dynamics of \eqref{ode epsilon} in the regime $U=\mathcal{O}(1)$. To that end, we consider the $U-$directional chart $K_1$ \eqref{blow-up in K1}. 
Applying \eqref{blow-up in K1} to \eqref{ode epsilon} gives the system 
\begin{equation}
    \begin{aligned}
        \dot{r}_1&=r_1\left(r_1^nw_1-\tfrac{1}{\sqrt{1+n}}\right),\\
        \dot{w}_1&=\tfrac{1}{\sqrt{1+n}}r_1\varepsilon_1+\left(\tfrac{1}{\sqrt{1+n}}-r_1^{1+n}\varepsilon_1\right)w_1-r_1^nw_1^2-(1-r_1^n)-n\left(r_1^nw_1-\tfrac{1}{\sqrt{1+n}}\right)w_1,\\
        \dot{\varepsilon}_1&=-(1+n)\varepsilon_1\left(r_1^nw_1-\tfrac{1}{\sqrt{1+n}}\right).
        \label{K_1}
    \end{aligned}
\end{equation}
Our goal for this subsection is to construct $\Gamma_1$, which is the portion of the singular orbit found in $K_1$.
First notice that \eqref{K_1} admits the line of equilibria  $\ell_1^-=\left\{\left(1, \tfrac{1}{\sqrt{1+n}}, \varepsilon_1\right)\;\bigg|\;\varepsilon_1\in[0, \varepsilon_0)\right\}$. For each fixed $r_1^{1+n}\varepsilon_1=\varepsilon$, $\ell_1^-$ corresponds to the point $Q^-$ before blow-up. As our goal is to construct a singular (in $\varepsilon$) heteroclinic orbit, we identify the point $Q_1^-=\left(1, \tfrac{1}{\sqrt{1+n}}, 0\right)$, obtained by taking $\varepsilon_1\to0$ on $\ell_1^-$. Furthermore, we identify the point $P_1=\left(0,\tfrac{1}{\sqrt{1+n}} ,0\right)$ which corresponds to a projection of $Q^0$ before blow-up. A straightforward calculation gives the following result.
\begin{lem} 
    The equilibrium $Q_1^-$ has the eigenvalues $\tfrac{n}{\sqrt{1+n}}, -\sqrt{1+n}, 0$. The corresponding eigenvectors are $(1,0,0)^T$, $\left(1,-\tfrac{1+2n}{\sqrt{1+n}},0\right)^T$, $(0,0,1)^T$.
    The equilibrium $P_1$ has the eigenvalues $-\tfrac{1}{\sqrt{1+n}}, \sqrt{1+n}, \sqrt{1+n}$. The corresponding eigenvectors are $(1,0,0)^T,(0,1,0)^T, (0,0,1)^T$.
\end{lem}
As $\varepsilon=r_1^{1+n}\varepsilon_1$, there are two limiting systems, namely $\varepsilon_1\to0$ and $r_1\to0$. We first take $\varepsilon_1\to 0$ in \eqref{K_1} to obtain 
\begin{equation}
    \begin{aligned}
        \dot{r}_1&=r_1\left(r_1^n w_1-\tfrac{1}{\sqrt{1+n}}\right),\\
        \dot{w}_1&=\sqrt{1+n}\;w_1-1+r_1^n(1-(1+n)w_1^2).
        \label{K_1 eps_1=0}
    \end{aligned}
\end{equation}
It is obvious that the invariant line $w_1(r_1)=\tfrac{1}{\sqrt{1+n}}$, which we call $\Gamma_1^-$, lies in both $W_1^{\rm u}(Q_1^-)$ and $W_1^{\rm s}(P_1)$. 
The second limiting system is obtained by taking $r_1\to0$ in \eqref{K_1} which yields
\begin{equation}
    \begin{aligned}
        \dot{w}_1&=\sqrt{1+n}w_1-1,\\
        \dot{\varepsilon}_1&=\sqrt{1+n}\varepsilon_1.
        \label{K_1, r_1=0}
    \end{aligned}
\end{equation}
The second portion of the singular orbit in chart $K_1$, which we call $\Gamma_1^+$, must be backward asymptotic to $P_1$ and can be computed by solving 
$\tfrac{d {w}_1}{d{\varepsilon}_1}=\tfrac{\sqrt{1+n}w_1-1}{\sqrt{1+n}\varepsilon_1}$ to obtain $w_1(\varepsilon_1)=\tfrac{1}{\sqrt{1+n}}+\sigma\varepsilon_1$, where $\sigma$ is a constant of integration. To determine $\sigma$, we recall that a perturbation of $W^{\rm u}(Q^-)$ is given by \eqref{perturbed orbit}.  Applying \eqref{blow-up in K1} to \eqref{perturbed orbit}, we find that 
\begin{equation}
    w_1(r_1, \varepsilon_1)=\tfrac{1}{\sqrt{1+n}}+\varepsilon_1
    \frac{\displaystyle \int_{r_1}^1 (1-s^n)^{\tfrac{1+n}{n}}ds}{(1-r_1^n)^{\tfrac{1+n}{n}}}+\mathcal{O}(r_1^{2+n}\varepsilon_1^2), 
    \label{w_1(r_1, eps_1)}
\end{equation}  
then taking $r_1\to0$ gives $w_1(0,\varepsilon_1)=\tfrac{1}{\sqrt{1+n}}+\alpha(n)\varepsilon_1$,
and therefore, $\sigma=\alpha(n)$.

We introduce the section $\Sigma_{1}^{\rm out}=\{(r_1, w_1, \rho)\;|\;(r_1, w_1)\in[0, r_0]\times[0, w_0]\}$, where $\rho$ is a fixed small constant, to track $\Gamma_1=\Gamma_1^-\cup\Gamma_1^+$ as it leaves the chart $K_1$. The exit point from the chart $K_1$ is given by $P_1^{\rm out}=\Gamma_1^+\cap\Sigma_1^{\rm out}=\left(0, \tfrac{1}{\sqrt{1+n}}+\alpha(n)\rho,  \rho\right)$.
The geometry and dynamics in chart $K_1$ for $n=1$ are shown in Figure~\ref{fig: K_1}.

\begin{figure}
    \centering
    \includegraphics[scale = 1.2]{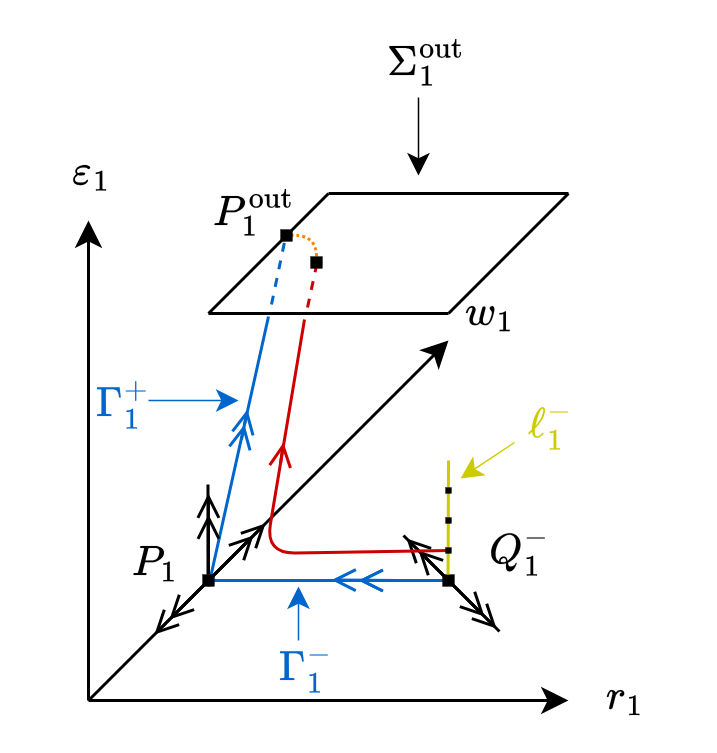}
    \caption{The geometry in chart $K_1$ for $n=1$.}
    \label{fig: K_1}
\end{figure}

\subsection{Rescaling chart $K_2$}
In this subsection, we study the dynamics of \eqref{ode epsilon} in the regime $U=\mathcal{O}(\sqrt[1+n]{\varepsilon})$ and $W=\mathcal{O}(\sqrt[1+n]{\varepsilon^n})$. To that end, we consider the rescaling chart $K_2$ \eqref{blow-up in K2}.
Applying \eqref{blow-up in K2} to \eqref{ode epsilon} gives 
\begin{equation}
    \begin{aligned}
        \dot{u}_2&=u_2\left(r_2^nw_2-\tfrac{1}{\sqrt{1+n}}\right),\\
        \dot{w}_2&=\tfrac{1}{\sqrt{1+n}}r_2+\left(\tfrac{1}{\sqrt{1+n}}-r_2^{1+n}\right)w_2-r_2^nw_2^2-u_2^n(1-r_2^nu_2^n),\\
        \dot{r}_2&=0.
        \label{K_2}
    \end{aligned}
\end{equation}
Our goal for this subsection is to construct the continuation of $\Gamma_1$ in the chart $K_2$. We shall call this orbit $\Gamma_2$. Notice that \eqref{K_2} admits the line of equilibria $\ell_2^0=(0, -r_2, r_2)$. For each fixed $r_2^{1+n}=\varepsilon$, $\ell_2^0$ corresponds to the point $Q^0$ before blow-up. Furthermore, we identify the point $Q_2^0$ obtained by taking $r_2\to 0$ on $\ell_2^0$. An easy calculation reveals the following.
\begin{lem}
    The origin $Q^0_2=(0,0,0)$ has the eigenvalues $-\tfrac{1}{\sqrt{1+n}}$, $\tfrac{1}{\sqrt{1+n}}, 0$. For $n=1$
    the corresponding eigenvectors are  $\left(1, \tfrac{1}{\sqrt{2}}, 0\right)^T$, $(0,1, 0)^T$, $(0, -1, 1)^T$; for $n>1$, the eigenvectors are $(1,0,0)^T$, $(0,1,0)^T$ and $(0,-1,1)^T$. 
    \label{lem: K_2 linear}
\end{lem}
In this chart, there is only one singular limit to consider, as $\varepsilon=r_2^{1+n}$. Taking $r_2\to0$ in \eqref{K_2} gives 
\begin{equation}
    \begin{aligned}
        \dot{u}_2&=-\tfrac{1}{\sqrt{1+n}}u_2,\\
        \dot{w}_2&=\tfrac{1}{\sqrt{1+n}}w_2-u_2^n.
        \label{K_2 SL}
    \end{aligned}
\end{equation}

Writing $\tfrac{dw_2}{du_2}=\sqrt{1+n}\;u_2^{n-1}-\tfrac{w_2}{u_2}$ gives $w_2(u_2)=\tfrac{u_2^n}{\sqrt{1+n}} +\tfrac{\beta}{u_2}$, which we will call $\Gamma_2$. Next, we introduce the following section 
$\Sigma_2^{\rm in}=\left\{\left(\tfrac{1}{\sqrt[1+n]{\rho}}, w_2, r_2\right)\:\bigg|\;(w_2, r_2)\in[0, w_0]\times[0, r_0]\:\right\},$ where 
$\Sigma_{2}^{\rm in}=\kappa_{12}(\Sigma_{1}^{\rm out})$.
In order to determine the constant $\beta$, we compute $P_2^{\rm in}=\kappa_{12}(P_1^{\rm out})=\left(\tfrac{1}{\sqrt[1+n]{\rho}}, \tfrac{1}{\sqrt[1+n]{\rho^n}} \left(\tfrac{1}{\sqrt{1+n}}+\alpha(n)\rho\right), 0\right)$. Taking $\beta=\alpha(n)$ ensures that $\Gamma_1$ and $\Gamma_2$ connect continuously.  

Lemma~\ref{lem: K_2 linear} shows that $Q_2^0$ is a saddle  for $n\geq 1$, and therefore there exists an invariant stable manifold $W_2^{\rm s}(Q_2^0)$, which for $r_2=0$ is given by $w_2(u_2)=\tfrac{u_2^n}{\sqrt{1+n}}$. $W_2^{\rm s}(Q_2^0)$ is a separatrix for trajectories in chart $K_2$. Furthermore, the orbit $\Gamma_2$ lies strictly above $ W_2^{\rm s}(Q_2^0)$ as $\tfrac{\alpha(n)}{ u_2}+\tfrac{u_2^n}{\sqrt{1+n}}>\tfrac{u_2^n}{\sqrt{1+n}}$.

Orbits which intersect $\Sigma_2^{\rm in}$ and have $w_2-$coordinate\;$<\tfrac{1}{\sqrt{1+n}}\tfrac{1}{\sqrt[1+n]{\rho^n}}$ (which is the intersection of $W_2^{\rm s}(Q_2^0)$ and $\Sigma_2^{\rm in}$), will reach $w_2=0$ in finite time. After blow-down $W<0$ and then $V<-\tfrac{1}{\sqrt{1+n}}$, for all $\xi>0$, which contradicts our requirement for a smooth front solution, which is $V=0$ as $\xi\to\infty$.
However, as $w_2^{\rm in}=\tfrac{1}{\sqrt[1+n]{\rho^n}} \left(\tfrac{1}{\sqrt{1+n}}+\alpha(n)\rho\right)>\tfrac{1}{\sqrt{(1+n)}}\tfrac{1}{\sqrt[1+n]{\rho^n}}$, then on $\Gamma_2$, $w_2>\tfrac{u_2^n}{\sqrt{1+n}}$ for all $u_2>0$.

To complete this subsection, we introduce the section 
$\Sigma_2^{\rm out}=\left\{\left(u_2, \tfrac{1}{\delta}, r_2\right)\; \bigg|\:(u_2, r_2)\in[0, u_0]\times[0, r_0]\right\},$ to track $\Gamma_2$ as it leaves the chart $K_2$. The constant $\delta>0$, is a fixed, small constant. We find $P_2^{\rm out}=\Gamma_2\cap\Sigma_2^{\rm out}=\left(u_2^{\rm out}(\delta), \tfrac{1}{\delta}, 0\right)$, where $u_2^{\rm out}(\delta)$ is the smaller root of  $\tfrac{1}{\delta}=\tfrac{1}{\sqrt{1+n}}\left(u_2^{\rm out}(\delta)\right)^n+\tfrac{\alpha(n)}{u_2^{\rm out}(\delta)}$. Note that $u_2^{\rm out}=\alpha(n)\delta+\mathcal{O}(\delta^{2+n})$.
The geometry and dynamics in chart $K_2$ for $n=1$ are shown in Figure~\ref{fig: K_2}.
\begin{figure}
    \centering
    \includegraphics[scale = 1.2]{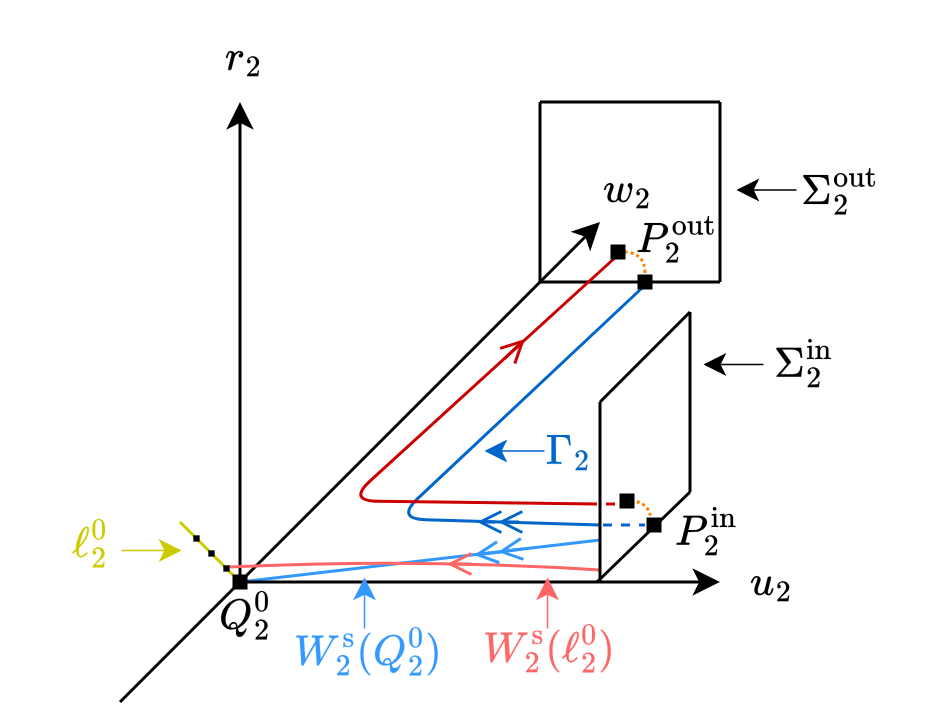}
    \caption{The geometry in chart $K_2$ for $n=1$.}
    \label{fig: K_2}
\end{figure}

We summarise this subsection by highlighting the key idea, which is that $\Gamma_2$ lies strictly above the stable separatrix $W^s(Q_2^0)$.
Thus, $\Gamma_2$ lies in the part of phase-space that exits the rescaling chart towards $K_3$, where it can match to the leading-edge dynamics of the smooth front. In the next subsection, we complete the construction of the singular orbit. 

\subsection{$W-$directional chart $K_3$}
In this subsection, we study the dynamics of \eqref{ode epsilon} in the regime $W=\mathcal{O}(1)$. To that end, we consider the $W-$directional chart $K_3$ \eqref{blow-up in K3}.
Applying \eqref{blow-up in K3} to \eqref{ode epsilon} gives 
\begin{equation}
    \begin{aligned}
        \dot{u}_3&=u_3\left(r_3^n-\tfrac{1}{\sqrt{1+n}}-\tfrac{1}{n}F(u_3, r_3, \varepsilon_3)\right),\\
        \dot{r}_3&=\tfrac{r_3}{n} F(u_3, r_3, \varepsilon_3),\\
        \dot{\varepsilon}_3&=-\tfrac{1+n}{n}\varepsilon_3 F(u_3, r_3, \varepsilon_3),
        \label{K_3}
    \end{aligned}
\end{equation}
where $F(u_3, r_3, \varepsilon_3)=\tfrac{1}{\sqrt{1+n}}-r_3^n+\tfrac{1}{\sqrt{1+n}}r_3\varepsilon_3-r_3^{1+n}\varepsilon_3-u_3^n(1-r_3^nu_3^n)$.
Our goal for this subsection is to complete the singular orbit, constructing the final part $\Gamma_3$.
Initially, we note that \eqref{K_3} admits the line of equilibria $\ell_3^+=\left\{\left(0, \tfrac{1}{\sqrt[2n]{1+n}}, \varepsilon_3\right)\;\big|\varepsilon_3\in[0, \sqrt[2n]{(1+n)^{1+n}}\;\varepsilon_0)\right\}$. For each fixed $r_3^{1+n}\varepsilon_3=\varepsilon$, $\ell_3^+$ corresponds to the point $Q^+$ before blow-up. Furthermore, we identify the point $Q_3^+=\left(0, \tfrac{1}{\sqrt[2n]{1+n}}, 0\right)$ obtained by taking $\varepsilon_3\to0$ on $\ell_3^+$. Additionally, we note that \eqref{K_3} admits a second relevant equilibrium $P_3=(0,0,0)$, which lies over the degenerate point $(U,W,\varepsilon)=(0,0,0)$ in the blown-up space. An easy calculation reveals the following.
\begin{lem}
    The equilibrium $Q_3^+$ has eigenvalues $0, -\tfrac{1}{\sqrt{1+n}}, 0$. For $n=1$, the corresponding eigenvectors are $(1, -1, 0)^T$, $(0,1,0)^T$ and $(0,0,1)^T$; for $n>1$, the eigenvectors are $(1,0,0)^T, (0,1,0)^T$ and $(0,0,1)^T$.
    The equilibrium $P_3$ has the eigenvalues $-\tfrac{\sqrt{1+n}}{n},\tfrac{1}{n\sqrt{1+n}},-\tfrac{\sqrt{1+n}}{n}$. The corresponding eigenvectors are $(1,0,0)^T, (0,1,0)^T$ and $(0,0,1)^T$.
    \label{lem: K_3 lin}
\end{lem}
As in the chart $K_1$, there are two limiting systems to consider, as $r_3^{1+n}\varepsilon_3=\varepsilon$.
We first take $r_3\to0$, and therefore $F(u_3,0, \varepsilon_3)=\tfrac{1}{\sqrt{1+n}}-u_3^n$, so that \eqref{K_3} becomes
\begin{equation}
    \begin{aligned}
        \dot{u}_3&=\tfrac{u_3}{n}\left(u_3^n-\sqrt{1+n}\right),\\
        \dot{\varepsilon}_3&=-\tfrac{1+n}{n}\varepsilon_3\left(\tfrac{1}{\sqrt{1+n}}-u_3^n\right).
        \label{K_3 r_3=0}
    \end{aligned}
\end{equation}
To solve for $\Gamma_3^-$, we write $\tfrac{d\varepsilon_3}{du_3}=\tfrac{\dot{\varepsilon}_3}{\dot{u}_3}$, giving $\varepsilon_3(u_3)=\gamma u_3\left(1-\tfrac{u_3^n}{\sqrt{1+n}}\right)$.

We introduce the section $\Sigma_3^{\rm in}=\{(u_3, r_3, \delta^{\tfrac{1+n}{n}})|\;(u_3, r_3)\in[0, u_0]\times[0,r_0]\}$, where $\kappa_{23}(\Sigma_2^{\rm out})=\Sigma_3^{\rm in}$, and $\kappa_{23}(P_2^{\rm out})=P_3^{\rm in}=\left(\delta^{\tfrac{1}{n}}u_2^{\rm out}(\delta),
0,
\delta^{\tfrac{1+n}{n}}\right)$, with $u_2^{\rm out}(\delta)$ defined in the previous section.

The matching between $\Gamma_2$ and $\Gamma_3^-$ takes place in the section $\Sigma_3^{\rm in}$, where $\varepsilon_3=\delta^{\tfrac{1+n}{n}}$, therefore, from the definition of $u_2^{\rm out}(\delta)$, $\gamma=\tfrac{1}{\alpha(n)}$. Then, we define the orbit $\Gamma_3^{-}$ by  $\varepsilon_3(u_3)=\tfrac{u_3}{\alpha(n)}\left(1-\tfrac{1}{\sqrt{1+n}}u_3^n\right)$. 
The second limiting system is obtained by taking $\varepsilon_3\to0$ in \eqref{K_3} to obtain 
\begin{equation}
    \begin{aligned}
        \dot{u}_3&=\tfrac{u_3}{n}\left( (1+n)r_3^n-\sqrt{1+n}+u_3^n(1-r_3^nu_3^n)\right),\\ \dot{r}_3&=\tfrac{r_3}{n}\left(\tfrac{1}{\sqrt{1+n}}-r_3^n-u_3^n(1-r_3^nu_3^n)\right).
        \label{K_3 eps_3=0}
    \end{aligned}
\end{equation}
The next result follows from centre-manifold theory and the invariance equation.
\begin{lem}
    The equilibrium $Q_3^+$ admits a centre-manifold $W^{\rm c}(Q_3^+),$ which can be written as a graph $r_3(u_3, \varepsilon_3)$, where $r_3(u_3, 0)=\tfrac{1}{\sqrt[2n]{1+n}}-\tfrac{\sqrt[2n]{(1+n)^{n-1}}}{n}u_3^n+\mathcal{O}(u_3^{1+n})$.
\end{lem}
The curve $u_3(r_3)=0$, which we call $\Gamma_3^{+}$, lies in $W^{\rm u}(P_3)$ and $W_{3}^{\rm s}(Q_3^+)$ and is invariant under the flow of \eqref{K_3 eps_3=0}. We note that in the singular limit $\varepsilon_3=0$, the convergence to $Q_3^+$ follows the stable hyperbolic direction, whereas, for $\varepsilon>0$ ($\varepsilon_3>0$), convergence to $\ell_3^+$ can only be obtained through the centre-manifold $W^{\rm c}(\ell_3^+)$. 
The geometry and dynamics in chart $K_3$ for $n=1$ are shown in Figure~\ref{fig: K_3}.

\begin{figure}[H]
    \centering
    \includegraphics[scale = 1.2]{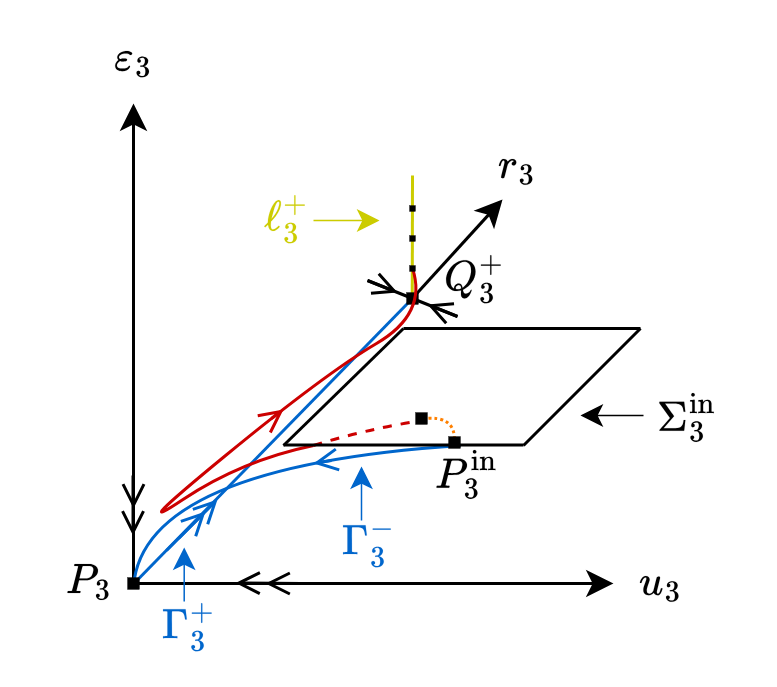}
    \caption{The geometry in chart $K_3$ for $n=1$.}
    \label{fig: K_3}
\end{figure}

\subsection{Proof of Theorem~\ref{thm:1}}
We summarise the previous three subsections with the following result.
\begin{prop}
    For \eqref{K_1}, \eqref{K_2} and \eqref{K_3} there exists a singular heteroclinic orbit $\overline{\Gamma}$, connecting $Q_1^-$ and $Q_3^+$.
\end{prop}
\begin{proof}
In $K_1$, the orbit $\Gamma_1^-$ connects $Q_1^-$ to $P_1$, while $\Gamma_1^+$ leaves $P_1$ and intersects $\Sigma_1^{\rm out}$ at $P_1^{\rm out}$. The transition map $\kappa_{12}$ maps this point to
$P_2^{\rm in}$, which determines the orbit $\Gamma_2$ in $K_2$. This orbit intersects $\Sigma_2^{\rm out}$ at $P_2^{\rm out}$. Applying $\kappa_{23}$ gives the point $P_3^{\rm in}$ on $\Gamma_3^-$. The orbit $\Gamma_3^-$ then connects to
$P_3$, and $\Gamma_3^+$ connects $P_3$ to $Q_3^+$. Together these orbits form the singular heteroclinic orbit $\overline{\Gamma}$.
\end{proof}

Next, we consider the persistence of $\overline{\Gamma}$ for sufficiently small $\varepsilon>0$.
\begin{prop}
    For every sufficiently small $\varepsilon>0$, \eqref{ode epsilon} admits a (locally) unique heteroclinic orbit connecting $Q^-$ and $Q^+$.
\end{prop}
\begin{proof}
    In chart $K_1$, the branch of the unstable manifold of $Q_1^-$ selected by the singular orbit is a perturbation of $\Gamma_1$ given by \eqref{w_1(r_1, eps_1)}. Furthermore, it is the unique orbit which is backward asymptotic to $Q_1^-$. It follows that for $\varepsilon>0$, $\hat{P}_1^{\rm out}=\left(\sqrt[1+n]{\tfrac{\varepsilon}{\rho}},\tfrac{1}{\sqrt{1+n}}+\alpha(n)\rho
    \left[
    1-\tfrac{1}{\alpha(n)}
    \sqrt[1+n]{\left(\tfrac{\varepsilon}{\rho}\right)}
    +\tfrac{1+n}{n}
    \sqrt[1+n]{\left(\tfrac{\varepsilon}{\rho}\right)^n}
    +\mathcal{O}(\varepsilon)
    \right],\rho  \right)$ (which is the perturbation of $P_1^{\rm out}$), and therefore, $\hat{P}_2^{\rm in}=\kappa_{12}(\hat{P}_1^{\rm out})=\left(\tfrac{1}{\sqrt[1+n]{\rho}},\tfrac{1}{\sqrt[1+n]{\rho^n}}\left[\tfrac{1}{\sqrt{1+n}}+\alpha(n)\rho
    \left[
    1-\tfrac{1}{\alpha(n)}
    \sqrt[1+n]{\left(\tfrac{\varepsilon}{\rho}\right)}
    +\tfrac{1+n}{n}
    \sqrt[1+n]{\left(\tfrac{\varepsilon}{\rho}\right)^n}
    +\mathcal{O}(\varepsilon)
    \right]\right],\sqrt[1+n]{\varepsilon}  \right)$.

    In chart $K_2$, recall that in the singular limit ($r_2=0$), the stable manifold of $Q_2^0$, $W_2^{\rm s}(Q_2^0)$ forms a separatrix for trajectories in $K_2$. Since $\Gamma_2$ lies a positive distance above $W_2^{\rm s}(Q_2^0)$ and since the stable manifold of the nearby saddle $\ell_2^0=(0, -\sqrt[1+n]{\varepsilon}, \sqrt[1+n]{\varepsilon})$ depends smoothly on $r_2$, the perturbed trajectory remains above $W_2^{\rm s}(\ell_2^0)$ for all sufficiently small $r_2>0$.
    As a result, the exit point from chart $K_2$ $\hat{P}_2^{\rm out}$ is a regular perturbation of the same point obtained in the singular limit $P_2^{\rm out}$.  
    In chart $K_3$ recall that the singular orbit $\Gamma_3=\Gamma_3^-\cup\Gamma_3^+$, passes through the equilibrium $P_3$. For $\varepsilon>0$ the standard local passage near a hyperbolic saddle (see for example \cite{POPOVIC20121976}) implies that the perturbation of $\Gamma_3=\Gamma_3^-\cup\Gamma_3^+$ enters a neighbourhood of $\ell_3^+$, close to the singular orbit $\Gamma_3^+$. Finally, recall from Lemma~\ref{lem: K_3 lin} that $\ell_3^+$ has a one-dimensional stable manifold and a two-dimensional centre-manifold, where one of these centre directions is due to $\ell_3^+$ being a line of equilibria. Then, convergence to $\ell_3^+$ is obtained by reducing the dynamics of \eqref{K_3} to the centre-manifold $r_3=\tfrac{1}{\sqrt[2n]{1+n}}-\tfrac{\sqrt[2n]{(1+n)^{n-1}}}{n(1+\sqrt{1+n}\;\varepsilon)}u_3^n+\mathcal{O}(u_3^{1+n})$ defined on the invariant surface $r_3^{1+n}\varepsilon_3=\varepsilon$, where $\dot{u}_3=-\tfrac{u_3^{1+n}}{1+\sqrt{1+n}\;\varepsilon}+\mathcal{O}(u_3^{2+n})$.
\end{proof}
This completes the proof of Theorem~\ref{thm:1}.

\section{Discussion}
In this paper, we have proven that the family of reaction-diffusion equations \eqref{pde1} admits smooth front solutions for wave speeds $c=\tfrac{1}{\sqrt{1+n}}+\varepsilon$, with $\varepsilon$ sufficiently small and $n$ a positive integer. The key to our proof is that the correct blow-up scaling \eqref{blow-up} reveals the special singular orbit $\Gamma_2$. The central observation is that $\Gamma_2$ lies strictly above the stable manifold separatrix $W_2^{\rm s}(Q_2^0)$ (which corresponds to the sharp front solutions), and therefore exits into chart $K_3$, where $W=r_3>0$. From here, smooth wave front solutions can be obtained.

Unlike the method of matched asymptotics used by Sherratt \cite{SHERRATT199633,Sherratt2010} in the $n=1$ case, we do not calculate a composite solution for $u(x-ct)$; however, using a normal-form calculation in a neighbourhood of the resonant saddles $P_1$ and $P_3$ 
the leading-order asymptotics for $U(\xi)$ can be obtained. Furthermore, due to the simple nature of the vector fields in the charts, an exact singular orbit can be obtained, leading to a computationally simple proof of the existence of smooth fronts. The benefit of our method is that we are able to treat the whole family of equations \eqref{pde1} in a single calculation. We note that a matched asymptotics approach to this problem for general positive integers $n$ may not be computationally feasible.

A natural direction for future work is to quantify the difference between the smooth
front ($\varepsilon>0$) and the sharp front ($\varepsilon=0$). For $n=1$, it was shown
in \cite{Sherratt2010} that the maximum difference between the smooth-
and sharp-front waves is $\mathcal{O}(\varepsilon\log\varepsilon)$ as $\varepsilon\to 0$.
This difference is caused by logarithmic switchback terms \cite{NPopovic_2005}, which are characteristic of
singular perturbation problems. In our geometric framework, such terms arise naturally
from the passage near the resonant saddles $P_1$ and $P_3$, and so our approach may be
well suited to recovering and generalising this result to all $n\geq 1$.

More broadly, our approach may be applicable to travelling fronts with genuine jump
discontinuities, or ``shock-fronted'' waves. Such solutions arise in models of chemotactic and diffusive cellular migration, where smooth and discontinuous travelling waves exist \cite{Landman2003, Landman2005}. As in the present work, the discontinuity in these problems reflects a degeneracy in the diffusive flux. Geometric desingularisation may offer a systematic way to resolve these discontinuities and to characterise the 
smooth-to-shock transitions.

\section*{Acknowledgements}
The author would like to thank Nikola Popovi\'c for numerous helpful discussions, and
Jonathan Sherratt for providing valuable feedback. The author was supported by the EPSRC Centre for Doctoral Training in Mathematical Modelling, Analysis and Computation (MAC-MIGS) funded by the UK Engineering and Physical Sciences Research Council (EPSRC) grant EP/S023291/1, Heriot-Watt University, and the University of Edinburgh.

\addcontentsline{toc}{section}{References}
\nocite{*}
\printbibliography
\end{document}